%% file: root.tex
\documentclass[letterpaper, 10 pt, conference]{ieeeconf}  
\input{Packages}
\input{CustomCommands}

\IEEEoverridecommandlockouts                              

\title{\LARGE \bf
Improved Small-Signal $\Ell_2$-gain Analysis for Nonlinear Systems
}

\author{Amy K. Strong$^{1}$, Reza Lavaei$^{1}$, and Leila J. Bridgeman$^1$
\thanks{This work was supported by NSF GFRP Grant No. 1644868, the Alfred P. Sloan
Foundation, ONR YIP Grant No. N00014-23-1-2043, and NSF Grant No. 2303158.}
\thanks{$^{1}$Amy K. Strong,  Reza Lavaei, and Leila J. Bridgeman are with the Department of Mechanical Engineering and Materials Science at Duke University, Durham, NC, 27708, USA.
        (email: {\tt\small amy.k.strong@duke.edu, reza.lavaei@duke.edu, leila.bridgeman@duke.edu}), phone: (919) 660-1260}}

\newcommand\copyrighttext{%
  \footnotesize \textcopyright 2024 IEEE.  Personal use of this material is permitted.  Permission from IEEE must be obtained for all other uses, in any current or future media, including reprinting/republishing this material for advertising or promotional purposes, creating new collective works, for resale or redistribution to servers or lists, or reuse of any copyrighted component of this work in other works.}
\newcommand\copyrightnotice{%
\begin{tikzpicture}[remember picture,overlay]
\node[anchor=south,yshift=10pt] at (current page.south) {\fbox{\parbox{\dimexpr\textwidth-\fboxsep-\fboxrule\relax}{\copyrighttext}}};
\end{tikzpicture}%
}

\begin{document}
\bstctlcite{BSTcontrol}

\maketitle
\copyrightnotice
\thispagestyle{empty}
\pagestyle{empty}

\begin{abstract}

The $\Ell_2$-gain characterizes a dynamical system's input-output properties, 
but can be difficult to determine for nonlinear systems. 
Previous work designed a nonconvex optimization problem to simultaneously search for a \ac{cpa} storage function and an upper bound on the small-signal $\Ell_2$-gain of a dynamical system over a triangulated region about the origin. This work improves upon those results by establishing a tighter upper-bound on a system's gain using a convex optimization problem.
By reformulating the relationship between the Hamilton-Jacobi inequality and $\Ell_2$-gain as a \ac{lmi} and then developing novel \ac{lmi} error bounds for a triangulation, tighter gain bounds are derived and computed more efficiently. Additionally, a combined quadratic and \ac{cpa} storage function is considered to expand the nonlinear systems this optimization problem is applicable to. Numerical results demonstrate the tighter upper bound on a dynamical system's gain. 

\end{abstract}

\section{INTRODUCTION}
\Ac{io} stability theory views a dynamical system as a mapping between inputs and outputs. One of the most widely used \ac{io} descriptors is the $\Ell_2$-gain of a system, which bounds the norm of the output with respect to that of the input. The $\Ell_2$-gain is leveraged in control design through the Small Gain Theorem \cite{zames1966input}, the basis of $\Hcal_{\infty}$ control. As such, the $\Ell_2$-gain of the system is an essential tool both for analysis of and control synthesis for dynamical systems.

While the gain of a linear dynamical system can be determined easily \cite{zhou1998essentials}, this is a more difficult task for nonlinear systems. The nonlinear Kalman-Yacubovich-Popov lemma \cite{brogliato2007dissipative} or relationships between $\Ell_2$-gain and the Hamilton-Jacobi equation \cite{van19922} can provide conditions to determine $\Ell_2$-gain. However, these methods require knowledge of a positive, semi-definite storage function for the system. Determining the storage function for a nonlinear system is a non-trivial task  and an active area of research\cite{giesl2015review}\cite{giesl2012construction}\cite{giesl2014revised}\cite{summers2013quantitative}\cite{mitchell2008flexible}.

Recent work formulated a nonconvex optimization problem that bounds the gain of a nonlinear system by  searching for a \acf{cpa} storage function that satisfies a \ac{hji} \cite{lavaei2023L2}. Constraints were imposed on the function at each vertex of a triangulation to ensure certain properties, leveraging a similar process in previous work on Lyapunov function synthesis \cite{giesl2012construction, giesl2014revised}. The optimization problem in \cite{lavaei2023L2} depends on a specific error term for the \ac{hji} constraint that, when imposed with the \ac{hji} on each vertex of the triangulation, ensures the inequality holds for the entire region. However, the \ac{hji} and its error term are polynomial in design variables -- leading to a conservative, nonconvex optimization problem that requires methods like \ac{ico} to solve. Moreover, specific requirements on the error bound restricted the optimization problem to a limited class of nonlinear systems -- notably precluding linear control affine terms. 

This paper develops a novel \ac{lmi} error upper bound to impose \ac{lmi} constraints on a \ac{cpa} function on a triangulation -- convexifying the previous optimization problem. With this, the \ac{hji} can be reformulated as an \ac{lmi}, creating a convex optimization problem that determines a \ac{cpa} storage function and the global optimum of the gain bound. Further, a convex optimization problem is developed for a combined quadratic and \ac{cpa} storage function, expanding the class of dynamical systems this work can be applied to. The end result is a tighter upper bound on the small-signal $\Ell_2$-gain of nonlinear dynamical systems than that of \cite{lavaei2023L2}. 

\section{PRELIMINARIES}

The interior, boundary, and closure of the set $\Omega \subset \mathbb{R}^n$ are denoted as $\Omega^o,$ $ \delta \Omega,$ and $\overline{\Omega},$ respectively. The symbol $\mathfrak{R}^n$ denotes the set of all compact subsets $\Omega \subset \mathbb{R}^n$ satisfying i) $\Omega^o$ is connected and contains the origin and ii) $\Omega = \overline{\Omega^o}$. Scalars, vectors, and matrices are denoted as $x,$ $\x,$ and $\X$, respectively. The notation $\mathbb{Z}_{a}^b$ ($\mathbb{Z}_{\bar{a}}^{\underline{b}}$) denotes the set of integers between $a$ and $b$ inclusive (exclusive). The $p$-norm of the vector $\x \vecdim{n}$ is shown as $||\cdot||_p,$ where $p \in \mathbb{Z}_1^{\infty}.$ Let $\Ell^{n}_p$ represent the normed function space with the norm $||\x||_{\Ell^{n}_p} = \left(\int_0^\infty ||\x(t)||^{p}dt\right)^{\frac{1}{p}}$ for $p \in \mathbb{Z}_1^{\underline{\infty}}$ and $||\x||_{\Ell^{n}_\infty} = \sup_{t\geq 0}||\x(t)||\leq \infty.$ The extended $\Ell^n_{pe}$ space is defined as the set of all functions $\x(t):[0, \infty) \rightarrow\mathbb{R}^n$ for which the truncation to $t \in [0, T]$ is in $\Ell^{n}_p$ $\forall \; T \geq 0.$
By $f \in \mathcal{C}^k$, it is denoted that a real valued function, $f$, is $k$-times continuously differentiable over its domain. 
 
The positive definiteness of a matrix is denoted by $\bP \succ 0$, while positive semi-definite and negative definite and semi-definite matrices are denoted similarly. Identity and zero matrices are shown as $\I$ and $\mathbf{0}.$ Let $1_n$ denote a vector of ones in $\mathbb{R}^n.$

The right-hand (left-hand) upper Dini derivatives for some function, $f$  is defined as $D^{+}f(\y) \defeq \lim \sup_{k \rightarrow 0^{+}} \!\!\frac{f(\x {+} kg(\x)) {-} f(\x)}{k}\! \left(\!\lim \sup_{k \rightarrow 0^{-}} \!\!\frac{f(\x {+} kg(\x)) {-} f(\x)}{k}\!\right)\!\!,$ where $k \in \mathbb{R}$ and $\dot{\x} = g(\x)$ \cite{giesl2014revised}.
\subsection{CPA Functions}
In \cite{giesl2012construction},\cite{giesl2014revised}, a compact subset of a system's domain,  $\mathbb{R}^n$, is triangulated, and a constrained linear optimization problem is formulated to solve for a \ac{cpa} Lyapunov function -- affine on each simplex.
This work uses a similar process to determine a positive, semi-definite storage function. The necessary definitions and tools for triangulation and problem formulation are listed below.
\begin{definition} (\textit{Affine independence} 
\cite{giesl2014revised}):
    A collection of $m$ vectors $\{\x_0, \x_1, \hdots , \x_m\} \subset \mathbb{R}^n$ is affinely independent if $\x_1-\x_0, \hdots, \x_m - \x_0$ are linearly independent.  
\end{definition}
\begin{definition} (\textit{$n$ - simplex} \cite{giesl2014revised}):
    A simplex, $\sigma$, is defined the convex hull of $n+1$ affinely independent vectors, $co\{\x_j\}_{j=0}^n$, where each vector, $\x_{j} \in \mathbb{R}^n$, is a vertex.
\end{definition}
\begin{definition} (\textit{Triangulation} \cite{giesl2014revised}):
    Let  $\Tcal = \{\sigma_i\}_{i=1}^{m_{\Tcal}} \in \mathfrak{R}^n$ represent a finite collection of $m_{\Tcal}$ simplexes, where the intersection of any two simplexes is a face or an empty set.
\end{definition}

Let $\Tcal = \{\sigma_i\}_{i=1}^{m_{\Tcal}}.$ Further, let $\{\x_{i,j}\}_{j=0}^n$ be $\sigma_i$'s vertices. The choice of $\x_{i,0}$ in $\sigma_i$ is arbitrary unless $0 \in \sigma_i,$ in which case $\x_{i,0} = 0$ \cite{giesl2014revised}. The vertices of the triangulation $\Tcal$ of $\Omega$ are denoted as $\mathbb{E}_{\Omega}.$
Let $\Tcal_{0}$ denote the simplexes in $\Tcal$ containing $0$ and $\Tcal_{\Omega\setminus\{0\}}$ denotes those in $\Omega$ that do not contain $0$.
\begin{lemma}\label{lemma_gradW}(Remark 9 \cite{giesl2014revised})
    Consider the triangulation $\Tcal = \{\sigma_i\}_{i=1}^{m_{\Tcal}},$ where $\sigma_i = \text{co}(\{\x_{i,j}\}_{j=0}^n)$, and a set $\W = \{W_{\x}\}_{\x \in \mathbb{E_{\Tcal}}}\subset\mathbb{R},$ where $W(\x) = W_{\x}, \forall \x \in \mathbb{E}_{\Tcal}.$ For simplex $\sigma_i$, let $\X_i \vecdim{n \times n}$ be a matrix that has $\x_{i,j} - \x_{i,0}$ as its $j$-th row and $\bar{W}_i\vecdim{n}$ be a vector that has $W_{\x_{i,j}} - W_{\x_{i,0}},$ as its $j$-th element. The function $W(\x) = \x_i^\top \X_i^{-1}\bar{W}_i,$ is the unique \ac{cpa} interpolation of $\W$ on $\Tcal$ for $\x \in \sigma_i$. 
\end{lemma}
The following lemma uses Taylor's Theorem \cite{fitzpatrick2009advanced} to develop an error term that compares function $g \in \mathcal{C}^2$ evaluated at some $\x \in \sigma_i$ to $g$ evaluated at the vertex points of $\sigma_i$. 
\begin{lemma}\label{lemma_eq_bound} (\textit{Proposition 2.2 and Lemma 2.3} \cite{giesl2012construction})
Consider $\hat{\Omega} \in \mathfrak{R}^n$ and its triangulation $\Tcal = \{\sigma_i\}_{i = 1}^{m_{\Tcal}}$. Let $g: \hat{\Omega} \rightarrow \mathbb{R}^n$ where $g \in \mathcal{C}^2.$ Then, for any $x \in \sigma_i = \text{co}(\{\x_{i,j}\}_{j=0}^n) \in \Tcal$,
\begin{equation}\label{eq:infNormError}
    \norm{g(\x) - \sum_{j=0}^n\lambda_jg(\x_{i,j})}_{\infty} \leq \frac{1}{2}\beta_i\sum_{j=0}^n\lambda_jc_{i,j},
\end{equation}
where $\{\lambda_j\}_{j=0}^n \in \mathbb{R}$ is the set of unique coefficients satisfying $\x = \sum_{j=0}^n\lambda_j\x_{i,j}$ with $\sum_{j=0}^n\lambda_j = 1$ and $0 \leq \{\lambda_j\}_{j=0}^n \leq 1,$
\begin{equation}\label{eq:oldC}
     c_{i,j} {=} n\norm{\x_{i,j} {-} \x_{i,0}}\!(\max_{k \in \mathbb{Z}_1^n}\norm{\x_{i,k} {-} \x_{i,0}}\! {+} \!\norm{\x_{i,j} {-} \x_{i,0}}),
\end{equation}
and 
\begin{equation*}\label{eq:oldB}
    \beta_i \geq \max_{p,q,r \in\mathbb{Z}_1^n} \max_{\xi \in \sigma_i} \absVal{\frac{\partial^2f^{(p)}(\x)}{\partial \x ^{(q)}\partial \x ^{(r)}}\Bigr|_{\x = \xi}}.
\end{equation*}
\end{lemma}
\subsection{\texorpdfstring{$\Ell_2$}{L2} Stability Analysis}
The $\Ell_2$-gain is a general \ac{io} descriptor of a mapping between two Hilbert spaces.
\begin{definition}(\textit{$\Ell_{p}$ stability \cite{zames1966input}})
    A mapping $\Gcal: \Ell^{m}_{pe} \rightarrow \Ell^{q}_{pe}$ is $\Ell_{p}$ finite gain stable if there exists $\gamma_1, \gamma_2 \geq 0$ such that
    \begin{equation}\label{eq:l2}
        \norm{(\Gcal \bu)_{\tau}}_{\Ell_{p}} \leq \gamma_1 \norm{\bu_{\tau}}_{\Ell_{p}} + \gamma_2,
    \end{equation}
    where $\bu \in \Ell^{m}_{pe}$ and $\tau \in [0, \infty).$
\end{definition}
\begin{definition}(\textit{Small-signal $\Ell_{p}$ stability \cite{khalil2002}})
The mapping $\Gcal: \Ell^{m}_{pe} \rightarrow \Ell^{q}_{pe}$ is $\Ell_{p}$ small-signal finite-gain stable if there exists $r_u>0$ such that (\ref{eq:l2}) is satisfied for all $\bu \in \Ell^m_{pe}$ with $\sup_{0 \leq t \leq \tau}\norm{\bu(t)} \leq r_{u}.$
\end{definition}

The \ac{hji} establishes a relationship between a system's $\Ell_2$-gain and the Hamilton-Jacobi equations \cite{van19922} and is the key to relating gain and \ac{cpa} storage functions in \cite{lavaei2023L2}.
\begin{theorem} (\cite{van19922})
    Consider the smooth system $\dot{\x} = f(\x) + g(\x)\bu, \quad \y = h(\x),$ where $\x \vecdim{n}$, $\y \vecdim{p}$, $\bu \vecdim{m},$ and $f(0) = h(0) = 0.$ 
    Let $\gamma > 0$ and suppose there is a smooth, positive semi-definite function $V: \mathbb{R}^n \rightarrow \mathbb{R}$ that satisfies the \ac{hji} (presented here in \ac{lmi} form),
    \begin{equation}\label{eq:gain_lmi}
        \bmat{\nabla V^\top f(\x) & \nabla V^\top g(\x) & h^\top (\x) \\
        * & -2\gamma^2\I & 0 \\
        * & * & -2\I} \preceq 0
    \end{equation}
    for all $\x \vecdim{n}.$ Then, for all $\x_0 \vecdim{n},$ the system is $\Ell_2$ stable with gain less than or equal to $\gamma.$
\end{theorem}
The storage function solution to (\ref{eq:gain_lmi}) is actually only required to be locally bounded and positive semi-definite, rather than smooth \cite[Theorem 3.1]{james1993partial}. 

To verify (\ref{eq:gain_lmi}), this paper synthesizes a \ac{cpa} storage function that is only defined on the bounded set $\Omega \in \mathfrak{R}^n.$  As established by \cite{lavaei2023L2}, the $\Ell_2$-gain can then only be found on a subset of $\Omega$, so the small-signal properties of the system are used to ensure the system is unable to leave the subset. This is accomplished by determining a modified \ac{cpa} barrier function (Theorem 2, \cite{lavaei2023L2}).

\section{MAIN RESULTS}

\acp{lmi} are a valuable tool used to solve complex control analysis or synthesis problems through optimization. This paper develops novel error bounds for an \ac{lmi} on an $n$-simplex, so that this tool can be leveraged for \ac{cpa}-analysis. Additionally, an \ac{lmi} to enforce constraints on a closed ball about the origin is developed to expand the type of systems \ac{cpa}-analysis can be applied to. These bounds are then used to design two convex optimization problems that bound the $\Ell_2$-gain for a given region of the state space. Both optimization problems search for storage functions that satisfy (\ref{eq:gain_lmi}), but the second problem uses a discontinuous quadratic and \ac{cpa} storage function to include systems with linear control affine terms (systems with a nonzero $\B$ matrix).

\subsection{LMI Error Bounds}

This section first develops a positive definite error bound matrix for an \ac{lmi} constraint applied to an $n-$simplex. Enforcing the \ac{lmi} constraint plus its error bound on the vertex points of an $n$-simplex ($\x_0,...,\x_n \in \sigma$) implies that the \ac{lmi} holds for all points within that simplex ($\x \in \sigma$). While \cite{giesl2012construction} established a general error bound for $\mathcal{C}^2$ vector-valued functions (\autoref{lemma_eq_bound}), it does not translate automatically to an \ac{lmi} that contains $\mathcal{C}^2$ vector-valued functions. The structure of an \ac{lmi} affects its definiteness and is considered in the following theorem.
\begin{theorem}\label{thm:lmi_bound}
Consider
\begin{equation}\label{eq:genericMat}
    \M(\x) = \bmat{\phi(\x) & \zeta^\top (\x) \\ \zeta(\x) & -\I},
\end{equation}
where $\x \vecdim{n}$, $\phi: \mathbb{R}^n \rightarrow \mathbb{R},$ $\zeta: \mathbb{R}^n \rightarrow \mathbb{R}^m$, $\phi,\zeta\in \mathcal{C}^2$, and $\zeta^{(k)}(\x)$ is the $k^{th}$ element of $\zeta$. Let $\sigma := \text{co}\{\x_{j}\}_{j=0}^n$ be an n-simplex in $\mathbb{R}^n$. If $\x = \sum_{i=0}^{n}\lambda_{j}\x_{j} \in \sigma,$ then
    \begin{align}\label{eq:lmi_bound}
        \M(\x) - \sum_{j=0}^{n}\lambda_{j}\M(\x_{j}) \preceq &  
         \sum_{{j} = 0}^n\lambda_{j}\!\!\bmat{\frac{1}{2}(\beta c_{j} {+} \!\!\sum\limits_{k=1}^m\!\!\mu_k^2 c_{j}^2) & 0 \\ * & \frac{1}{2}\I} \nonumber\\
         =& \sum_{{j} = 0}^n\lambda_j\E(\x_{j}) \defeqRight \E(\x),
    \end{align}
    where 
    \begin{equation}\label{eq:c}
        c_{j} = n\max_{\upsilon \in \mathbb{Z}_0^n}\norm{\x_{{j}} {-} \x_{\upsilon}}_2^2,
    \end{equation}
    \normalsize
    \begin{equation}
        \beta \geq \max_{q,r \in\mathbb{Z}_1^n, \xi \in \sigma} \absVal{\frac{\partial^2\phi(\x)}{\partial \x ^{(q)} \partial \x ^{(r)}}\Bigr|_{\x = \xi}}, \text{ and }
    \end{equation}
    \begin{equation}
         \mu_{k} \geq \max_{q,r \in\mathbb{Z}_1^n, \xi \in \sigma} \absVal{\frac{\partial^2\zeta^{(k)}(\x)}{\partial \x ^{(q)}\partial \x ^{(r)}}\Bigr|_{\x = \xi}}.
    \end{equation}
Moreover, if $\M(\x_{j}) +\E(\x_{j}) \preceq 0$ holds for all vertex points of $\sigma,$ then $\M(\x) \preceq 0$ for all $\x \in \sigma.$
\end{theorem}

\autoref{thm:lmi_bound} bounds the difference between $\M(\x)$ at any convex combination of vertex points
and the convex combination of the \ac{lmi} evaluated at each vertex point, $\sum_{{j}=0}^n\lambda_{j}\M(\x_{j})$. 
The proof parallels Proposition 2.2 in \cite{giesl2012construction} by developing remainder terms using Taylor's theorem, but exploits the structure of $\M(\x)$ to establish an \ac{lmi} error bound. Note the negative identity matrix in (\ref{eq:genericMat}) is essential to enforce negative semi-definiteness.

\begin{proof}
By definition, any point $\x \in \sigma$ can be written as a convex combination of the vertices, i.e., $\x = \sum_{{j}=0}^n\lambda_{j}\x_{j}$. Applying Taylor's Theorem \cite[Theorem 14.20]{fitzpatrick2009advanced} to $\phi$ and $\zeta$ about $\x$ for each vertex point in $\sigma$ results in 
\begin{align*}
    \sum_{{j}=0}^n\lambda_{j} \phi(\x_{j}) 
    {=} \sum_{{j}=0}^n\lambda_{j} \Bigl(&\phi(\x) {+} \langle \nabla \phi(\x),\Delta\x_{j}\rangle \\
    &{+} \langle \bH_{\phi}(\z_{{j},\phi})\Delta\x_{j},\Delta\x_{j} \rangle\Bigr)
\end{align*}
and
\begin{flalign*}
    \sum_{{j}=0}^n\lambda_{j} \zeta^{(k)}(\x_{j}) 
    {=}& \sum_{{j}=0}^n\lambda_{j} \Bigl(\zeta^{(k)}(\x) {+} \langle \nabla \zeta^{(k)}(\x),\Delta\x_{j}\rangle \\
    &{+} \frac{1}{2}\langle \bH_{\zeta^{(k)}}(\z_{{j},\zeta^{(k)}})\Delta\x_{j},\Delta\x_{j} \rangle\Bigr).
\end{flalign*}
Here, $\Delta\x_{j}=\x_{j} - \x$, $\bH_{\phi}$ is the Hessian for $\phi$, and $\z_{{j},\phi}$ and $\z_{{j},\zeta}$ are each some convex combination of $\x_{j}$ and $\x$. Because $\zeta$ is a vector-valued function, each dimension of $\zeta(\x) \vecdim{m}$ is separately expanded -- represented as element $\zeta^{(k)}(\x)$ with a corresponding Hessian $\bH_{\zeta^{(k)}}$ for $k =1,...,m$. 

Let $\tilde{\E}(\x){=} \M(\x) {-} \sum_{{j}=0}^{n}\lambda_{j}\M(\x_{j})$. In $\tilde{\E}(\x)$, the summed zeroth order terms of each function's Taylor expansion in $\sum_{{j}=0}^{n}\lambda_{j}\M(\x_{j})$ cancel with the corresponding terms in $\M(\x).$ Furthermore, the summed first order terms of each expansion become zero, because $\sum_{j = 0}^n\lambda_j\langle \nabla f(\x), \x_j - \x\rangle$ $= \langle \nabla f(\x), \sum_{j = 0}^n\lambda_j\x_j - \x\rangle$ $ = \langle \nabla f(\x),0 \rangle$. Altogether,
\begin{flalign*}
    \tilde{\E}(\x) 
    =& \frac{1}{2}\sum\limits_{{j}=0}^n\lambda_{j}
    \bmat{{-}\langle  \bH_{\phi}(\z_{{j},\phi})\Delta \x_{j}, \Delta \x_{j}\rangle\! &\!\! *\!& \!\!\hdots\!\! & \!*\\
    {-}\langle \bH_{{\zeta}^{(1)}}\!(\z_{{j},\zeta^{(1)}})\Delta \x_{j}, \Delta \x_{j}\rangle\! &\!\!0\! &\!\!\hdots\!\! & \!0\\
    \vdots &\! \vdots\! & \!\!\ddots\!\! & \!\vdots \\
    {-}\langle  \bH_{{\zeta}^{(m)}}\!(\z_{{j},\zeta^{(m)}})\Delta \x_{j}, \Delta \x_{j}\rangle\! & \!\!0\! &\!\!\hdots\!\!& \!0}\!.
\end{flalign*}

Consider the definition of negative definiteness. Let $\w \!\!= \!\!\bmat{\w_1 \!\!&\!\! \w_2^\top}\!^\top\!\!\!$, where $\w_1 \!\vecdim{1}$ and $\w_2\! \vecdim{m}\!.$ Then,
\begin{flalign*}\label{eq:bound2}
\begin{split}
    \w^\top \tilde{\E}(\x)\w = \sum_{{j}=0}^n&\!\lambda_{j}\!\Bigl[{-}\frac{1}{2}\w_1^\top\! \langle \bH_{\phi}(\z_{{j}, \phi})\Delta \x_{j}, \Delta \x_{j}\rangle \w_1\\
    &  {+}
    \w_2^\top\!\!\bmat{\!\!{-}\langle \bH_{{\zeta}^{(1)}}(\z_{{j},\zeta^{(1)}})\Delta \x_{j}, \Delta \x_{j}\rangle\!\\ \vdots \\ \!{-} \langle\bH_{{\zeta}^{(m)}}(\z_{{j},\zeta^{(m)}})\Delta \x_{j}, \Delta \x_{j}\rangle\!}\!\!\w_1\!\Bigr]\!.
\end{split}
\end{flalign*}
 Whether $\w^\top\tilde{\E}(\x)\w$ is positive or negative depends on both the values within $\tilde{\E}(\x)$ and $\w$ itself, because of cross terms. By completing the square on the cross terms $(2\w_1^\top\A^\top\B\w_2\leq\w_1^\top \A^\top\A\w_1 + \w_2^\top\B^\top\B\w_2)$, this dependency is removed to bound $\tilde{\E}(\x)$ for all $\w \vecdim{m+1}$:
\begin{flalign*}
    \begin{split}
        &\w^\top\tilde{\E}({\x})\w \leq
        \frac{1}{2}\!\sum_{{j}=0}^n\lambda_{j}\Bigl[ \w_1^\top\Bigl(\!{-}\langle   \bH_{\phi}(\z_{{j}, \phi})\Delta \x_{j}, \Delta \x_{j}  \rangle \ \\
        &{+} \sum_{k=1}^m\Bigl[\langle\bH_{{\zeta^{(k)}}}\!(\z_{{j},\zeta^{(k)}})\Delta \x_{j}, \Delta \x_{j}\rangle^2\Bigr]\Bigr)\w_1\Bigr] + \w_2^\top\frac{1}{2}\I\w_2.
    \end{split}
\end{flalign*}
The expression above can then be simplified by applying the Cauchy-Schwarz inequality, Lemma 2.3 of \cite{giesl2012construction}, and the bound $\norm{\Delta \x_{j}}_2^2 \leq \max_{\upsilon\in\mathbb{Z}_0^n}\norm{\x_{j} - \x_{\upsilon}}_2^2$ to produce the final upper bound on $\w^\top \tilde{\E}(\x)\w$,
\begin{equation*}
    \w_1^\top \frac{1}{2}\sum_{{j}=0}^n\lambda_{j}\Bigl( \beta c_{{j}} + \sum_{k=1}^m(\mu_k^2c_{{j}}^2)\Bigr)\w_1  + \w_2^\top \frac{1}{2}\I\w_2.
\end{equation*}
Hence, $\w^\top \tilde{\E}(\x)\w \leq \w^\top \E(\x)\w$ for all $ \w\vecdim{m+1}$, implying (\ref{eq:lmi_bound}). 

Now suppose that $\M(\x) \preceq 0$ must be imposed on all $\x \in \sigma.$ By assumption, $\M(\x_{j}) + \E(\x_{j}) \preceq 0$ holds for each vertex of $\sigma$  $(\x_0, \hdots, \x_n).$ The set of negative semi-definite \acp{lmi} is a convex cone \cite{boyd2004convex}. By enforcing $\M(\x_{j}) + \E(\x_{j}) \preceq 0$ on each vertex, the expression $\sum_{{j}=0}^n\lambda_{j}(\M(\x_{j}) + \E(\x_{j})) \preceq 0$ also holds. The \ac{lmi} $\M(\x) + \E(\x) \preceq 0$ implies $\M(\x)\preceq 0,$ because $\E(\x)\succeq 0$. Therefore, $\M(\x) \preceq 0$ for all $\x \in \sigma.$
\end{proof}
\begin{theorem}\label{thm:originLMI}
    Consider the inequality
    \begin{equation}\label{eq:genIneqal}
         \zeta(\x)^\top\zeta(\x)+ \frac{1}{2}\left(\x^\top \theta(\x) + \theta(\x)^\top\x\right) \leq 0,
    \end{equation}
    where $\x \vecdim{n}$, $\theta:\mathbb{R}^n\rightarrow\mathbb{R}^n,$ $\zeta:\mathbb{R}^n\rightarrow\mathbb{R}^m,$ $\theta, \zeta \in \mathcal{C}^2,$ $\theta(0)  =0, $ and $\zeta(0) = 0.$  Let $B_{\epsilon}(0)$ be a closed ball about the origin with radius $\epsilon,$
    \begin{align}
        \beta_{\epsilon} \geq & \max_{p,q,r \in\mathbb{Z}_1^n, \xi \in B_{\epsilon}(0)}\absVal{\frac{\partial^2\theta^{(p)}(\x)}{\partial \x ^{(q)}\partial \x ^{(r)}}\Bigr|_{\x = \xi}} \text{, }\\
        \mu_{\epsilon} \geq & \max_{p \in\mathbb{Z}_1^m, q,r \in\mathbb{Z}_1^n, \xi \in B_{\epsilon}(0)}\absVal{\frac{\partial^2\zeta^{(p)}{(\x)}}{\partial \x ^{(q)}\partial \x ^{(r)}}\Bigr|_{\x = \xi}},
    \end{align}
    and $\mathbf{J}_{\theta}(0)$ and $\mathbf{J}_{\zeta}(0)$ be the Jacobian of $\theta$ and $\zeta$, respectively, evaluated at $\x = 0.$ If 
    \begin{flalign}\label{eq:genEpLMI}
        &\M_{\epsilon} {=} \!\!\bmat{\frac{1}{2}(\mathbf{J}_{\theta}(0)^\top\!\!{+} \mathbf{J}_{\theta}(0){+}(\epsilon n^{\frac{3}{2}}\beta_{\epsilon} {+}\epsilon^2 n^2m\mu^2_{\epsilon})\I)\!& \!\mathbf{J}_\zeta (0)^{\top} \\ * \!&\! -\frac{1}{2}\I} \!\!\preceq \!0 
    \end{flalign}
    holds, then (\ref{eq:genIneqal}) holds for all $x \in B_{\epsilon}(0)$.
\end{theorem}
\begin{proof}
    For vector-valued functions, Taylor's theorem can be applied to each dimension. By applying Taylor's Theorem about the origin,
    \begin{flalign*}
    \zeta(\x)=&\mathbf{J}_{\zeta}(0)\x {+} \frac{1}{2}\!\!\!\bmat{\!\x^\top \bH_{\zeta^{(1)}}(\z_{\zeta^{(1)}}\!)\!\\ \vdots \\\x^\top\bH_{\zeta^{(m)}}(\z_{\zeta^{(m)}}\!)\!}\!\x {\defeqRight} (\A_1 {+} \frac{1}{2}\A_2)\x {\defeqRight} \A\x,\\
    \text{and}\\
    \theta(\x)=&\mathbf{J}_{\theta}(0)\x{+} \frac{1}{2}\!\!\!\bmat{\!\x^\top\bH_{\theta^{(1)}}(\z_{\theta^{(1)}})\!\\ \vdots \\ \x^\top\bH_{\theta^{(n)}}(\z_{\theta^{(n)}})\!}\!\x{ \defeqRight} (\B_1 {+} \frac{1}{2}\B_2)\x {\defeqRight} \B\x,
    \end{flalign*}
    where each $\z_{\zeta^{(k)}}$ and $\z_{\theta^{(r)}}$ are each some convex combination of $\x$ and $0$ for $k = 1,\hdots, m$ and $r = 1,\hdots, n.$
    Then, (\ref{eq:genIneqal}) is equivalently expressed as $\x^\top\A^\top\A\x + \black{\frac{1}{2}}(\x^\top\B\x{+}\black{\x^\top\B^\top\x}) \leq 0$. Noting that scalars equal their own transposes, factoring out $\x^\top$ and $\x$ and performing a Schur complement \cite{boyd2004convex} on $\A^\top\A$ results in the \ac{lmi}
    \begin{equation*}
        \tilde{\M}_{\epsilon}=\bmat{\black{\frac{1}{2}}(\B{+}\black{\B^\top)} & \A^\top \\ \A&-\I }\preceq 0.
    \end{equation*}
    Both $\A_2$ and $\B_2$ contain $\x$. Therefore, an infinite number of constraints are needed to enforce $\tilde{\M}_{\epsilon} \preceq 0$ for all $\x \in B_{\epsilon}(0).$
    
    By definition, the \ac{lmi} $\tilde{\M}_{\epsilon}\preceq 0$ is equivalent to $\w_1^\top \black{\frac{1}{2}(\B+\B^\top)}\w_1 {+} 2\w_1^\top\A^\top_1\w_2 {+}\w_1^\top\A^\top_2\w_2 {-} \w_2^\top\I\w_2 {\leq} 0$ holding for all $\w {=}\! \bmat{\w_1^\top & \w_2^\top}^\top,$ where $\w_1\!\!\in\!\!\mathbb{R}^{n}$ and $\w_2\!\!\in\!\! \mathbb{R}^{m}$. Like in \autoref{thm:lmi_bound}, problematic off-diagonal terms, in this case $\A_2,$ can be bounded above via Young's relation \cite{caverly2019lmi} to get the inequality $\w_1^\top\frac{1}{2}(\B {+}\black{\B^\top} {+}\A_2^\top\A_2)\w_1 {+} 2\w_1^\top \A^\top_1\w_2 {-} \w_2^\top\frac{1}{2}\I\w_2 \leq 0.$
    The terms $\A_2^\top\A_2$ and $\B_2$ are then bounded above using \black{the definition of the matrix two-norm and the  Cauchy-Schwarz inequality}. Lemma 2.3 from \cite{giesl2012construction} is then applied to the norm of each Hessian, and the definition of $B_{\epsilon}(0)$ is used to produce an upper bound on (\ref{eq:genIneqal}),
    \begin{flalign*}
        &\w_1^\top \frac{1}{2}\Bigl((\B_1{+}\black{\B_1^\top}) {+}(\epsilon n^{\frac{3}{2}}\beta_{\epsilon} +\epsilon^2 n^2 m\mu^2_{\epsilon})\I\Bigr)\w_1\\&{+} 2\w_1^\top\A^\top_1\w_2  {-} \w_2^\top\frac{1}{2}\I\w_2 \leq 0,
    \end{flalign*} which is equivalent to (\ref{eq:genEpLMI}). Then, $\M_{\epsilon} \preceq 0$ implies $\tilde{\M}_{\epsilon} \preceq 0,$ which is equivalent to (\ref{eq:genIneqal}).
\end{proof}

\subsection{\texorpdfstring{$\Ell_2$}{L2}-gain Analysis}
This section develops convex optimization problems to bound an input-affine system's $\Ell_2$-gain. Theorem \ref{thm:L2analysisOpt} covers when the input-affine term disappears at the origin, while Theorem \ref{thm:L2analysisOpt_B} allows it to remain nonzero at the origin. This distinction is important, as all error bounds in the optimization problem must be zero at the origin to prevent infeasability. While this always occurs when the input term disappears, a modified quadratic and \ac{cpa} storage function (which is quadratic at the origin) is needed for the nonzero case.

\begin{theorem}\label{thm:L2analysisOpt}
    Consider the constrained mapping $\Gcal: \Ell^{m}_{2e} \rightarrow \Ell^{p}_{2e}$ defined by $\y = \Gcal \bu,$
    \begin{flalign}\label{eq:dynSys}
        &\Gcal\!:\!
        \begin{cases}
            \dot{\x} {=} f(\x) {+} (\B {+} g(\x))\bu & \x \!\in\! \Xcal \!\in\! \mathfrak{R}^n\!, \bu \!\in\! \mathcal{U} \!\in\! \mathfrak{R}^m \\
            \y {=} h(\x),
        \end{cases}
    \end{flalign}
    where $f: \mathbb{R}^n \rightarrow \mathbb{R}^n$, $h: \mathbb{R}^n \rightarrow \mathbb{R}^p$, $\B \vecdim{n \times m}$,  $g$ is a $n\times m$ matrix where each $k$th column $g_k:\mathbb{R}^n \rightarrow \mathbb{R}^n$, and $f(0) = 0$, $g(0) = 0$, and $h(0) = 0.$ Let $\mathcal{U} = \{\bu \in \Ell^m_{2e} | \sup_{0 \leq t \leq \tau}\norm{\bu(t)}_{\infty} \leq r_u\}$ for some $r_u >0$ that ensures $\x$ remains in subset $\Omega$ of the state space \cite[Theorems 2, 3]{lavaei2023L2}. Let $\B = 0$. Suppose that $f, g, h \in \mathcal{C}^2$ for a triangulation, $\Tcal =\{\sigma_i\}_{i=1}^{m_{\Tcal}}$ of a set $\Omega \in \mathfrak{R}^n.$ Define a candidate \ac{cpa} storage function, $V = \{V_{\x}\}_{\x \in \mathbb{E}_{\Tcal}}.$ Consider the optimization problem
    \begin{equation*}
        \min_{\V, \alpha, \mathbf{L}} \alpha
    \end{equation*}
    \vspace{-6.5mm}
    \begin{subequations}\label{eq:optGamma}
        \begin{align}
            &\alpha > 0, \\
            &V_{\x} \geq 0 \quad \forall \x \in \mathbb{E}_{\Tcal}, \label{eq:constr_posDef}\\
            &\norm{\nabla V_i}_1 \leq l_i, \quad \forall i \in \mathbb{Z}_1^{m_{\Tcal}}, \\
            &\M_{i,j} \preceq 0, \quad \forall i \in \mathbb{Z}_1^{m_{\Tcal}}, \forall j \in \mathbb{Z}_0^n, x \neq 0,
        \end{align}
    \end{subequations}
     where $\alpha = \gamma^2$, $\mathbf{L} = \{l_i\}_{i=1}^{m_{\Tcal}} \subset \mathbb{R}^n,$ and
     \begin{flalign}\label{eq:gainUB}
     \begin{split} 
         &\M_{i,j} {=} \\&\!\!\!\!\!\!\!\!\!\!\bmat{\!\nabla V\!^\top\!  \!f(\x_{i,j}){+} \frac{1}{2}\!\Bigl(\!\beta_ic_{i,j}(1_n^\top l_i) {+} \!\!\!\sum\limits_{a=1}^p\!\!\rho_{i,a}^2c_{i,j}^2\!\Bigr)\!\!\!\!\!\!\!\!\!\!&\!\!\! *\!\!\! &\!\!\!* \!\!\!& \!\!\!*\!\!\! \\ g(\x_{i,j})^\top\!\nabla V\!\!\! &\!\!\!\!\! -2\alpha\I +\frac{1}{2}\I\!\!\!& \!\!\!*\!\!\! &\!\!\! * \!\!\!\\ h(\x_{i,j})\!\!\! &\!\!\! \0 \!\!\!& \!\!\!-\frac{3}{2}\I\!\!\! & \!\!\!*\!\!\! \\ (1_n^\top l_i)c_{i,j}\bmat{\mu_{i,1} \!\!& \!\hdots\! & \!\!\mu_{i,m}}^\top\! \!\!\!&\!\!\! \0\!\!\! &\!\!\! \0 \!\!\!& \!\!\!-2\I\!}\!\!.
    \end{split}
    \end{flalign}
    Further, $c_{i,j}$ is defined by (\ref{eq:oldC}) when $\sigma_i \in \Tcal_0$ and by (\ref{eq:c}) when $\sigma_i \in \Tcal_{\Omega\setminus\{0\}}$,
    \begin{align}\label{eq:beta}
        \beta_i \geq& \max_{p,q,r \in\mathbb{Z}_1^n, \xi \in \sigma_i}\absVal{\frac{\partial^2f^{(p)}(\x)}{\partial \x ^{(q)}\partial \x ^{(r)}}\Bigr|_{\x = \xi}}\\
        \label{eq:rho}
        \rho_{i,a} \geq& \max_{q,r \in\mathbb{Z}_1^n,\xi \in \sigma_i}\absVal{\frac{\partial^2h^{(a)}(\x)}{\partial \x ^{(q)}\partial \x ^{(r)}}\Bigr|_{\x = \xi}},\text{ and}\\
    \label{eq:mu}
        \mu_{i,k} \geq& \max_{p,q,r \in\mathbb{Z}_1^n, \xi \in \sigma_i}\absVal{\frac{\partial^2g^{(p)}_{k}(\x)}{\partial \x ^{(q)}\partial \x ^{(r)}}\Bigr|_{\x = \xi}},
    \end{align}
    where $a\in \mathbb{Z}_1^{p}$ represents each dimension of the output, and $k\in \mathbb{Z}_1^{m}$ represents each column of matrix $g(\x).$
    
    If the optimization problem, (\ref{eq:optGamma}), is feasible, then (\ref{eq:gain_lmi}) holds for all points in $\Omega^o$ and $\gamma^{*} = \sqrt{\alpha^*}$ is an upper bound on the $\Ell_2$-gain of $\Gcal$ in $\Omega^o.$
\end{theorem}
\begin{proof}
    Note that $\beta_i,$ $\rho_{i, a}$, and $\mu_{i, k}$ exist for all simplexes, because $f, g, h \in \mathcal{C}^2$ and $\Omega$ is bounded. Furthermore, $c_{i,j}$ is finite for all $n$-simplexes in $\Tcal,$ because $\Omega \in \mathfrak{R}^n.$ Defining $c_{i,j}$ as in (\ref{eq:oldC}) when $\sigma_i \in \Tcal_{0}$ ensures that the error bound is $0$ at the origin. 

    Constraint \ref{eq:gainUB} is a result of applying \autoref{thm:lmi_bound} to the gain \ac{lmi},  (\ref{eq:gain_lmi}), and then using the Schur complement. In detail, because $V$ is a \ac{cpa} function, $\nabla V$ is computed using Lemma \ref{lemma_gradW} and is constant over each $n-$simplex. When applying \autoref{thm:lmi_bound}, the error of each function in (\ref{eq:gain_lmi}) can factor $\nabla V$ out, e.g. $\nabla V^\top f(\x) - \nabla V^\top \sum_{i=0}^n\lambda_i f(\x_i) = \nabla V^\top(f(\x) -\sum_{i=0}^n\lambda_i f(\x_i)).$ 
    From \autoref{thm:lmi_bound}, the error in the off-diagonal terms of the \ac{lmi} is bounded above by completing the square.
    The Cauchy-Schwarz inequality is applied to error terms that interact with $\nabla V,$ e.g. $\frac{1}{2}\nabla V^\top \langle \bH_f(\z)\Delta \x_i, \Delta \x_i \rangle \leq \frac{1}{2}\norm{\nabla V}_2\norm{\bH_f(\z)}_2\norm{\Delta \x_i}_2^2$. Through equivalence of norms, $\norm{\nabla V}_2 \leq \norm{\nabla V}_1.$ The \ac{lmi} (\ref{eq:gain_lmi}) and its upper bound are then defined as
    \begin{align*}
            &\bmat{\nabla V^\top f(\x_{i,j})+\tilde{e}_{i,j} & \nabla V^\top g(\x_{i,j}) & h(\x_{i,j})\\ * & -2\alpha\I +\frac{1}{2} \I & 0 \\ * & * & -\frac{3}{2}\I}, \\
    \text{where} \\ 
            \tilde{e}_{i,j} {=} &\frac{1}{2}\Bigl(\beta_ic_{i,j}(1_n^\top l_i) + \sum\limits_{a=1}^p \rho_{i,a}^2c_{i,j}^2 + \sum\limits_{k=1}^m\mu_{i, k}^2c_{i,j}^2(1_n^\top l_i)^2\Bigr).
    \end{align*}
    Perform a Schur complement \cite{boyd2004convex} about 
    $\frac{1}{2}\sum\limits_{k=1}^m\mu_{i, k}^2c_{i,j}^2(1_n^\top l_i)^2$ to get the equivalent \ac{lmi}, (\ref{eq:gainUB}). 
    
    By enforcing Constraint \ref{eq:constr_posDef} on each $n-$simplex, $V(\x)$ is a viable storage function for $\Gcal$ in the region $\Omega$. From \autoref{thm:lmi_bound}, (\ref{eq:gain_lmi}) will hold for all $\x \in \Omega^o$ for $\gamma^*.$ 
\end{proof}

\begin{theorem}\label{thm:L2analysisOpt_B}
    Consider (\ref{eq:dynSys}). Suppose that $f, g, h \in \mathcal{C}^2$ for set $\Omega \in \mathfrak{R}^n.$ Let $\mathcal{U} {=} \{\bu \in \Ell^m_{2e} | \sup_{0 \leq t {\leq} \tau}\norm{\bu(t)}_{\infty} {\leq} r_u\}$ for some $r_u >0$ that ensures $\x$ remains in subset $\Omega$ of the state space \cite[Theorems 2, 3]{lavaei2023L2}. Let $B_\epsilon(0)$ be some ball about the origin with radius $\epsilon,$ and let $\widehat{\Tcal}\! =\{\sigma_i\}_{i=1}^{m_{\widehat{\Tcal}}}$ be a triangulation of $\Omega\setminus B_{\epsilon}(0)^{\circ}$. Define a candidate storage function,
    \begin{equation}\label{eq:newLyap_B}
        V(\x) = \begin{cases}
            \x^\top \bP \x & \norm{\x}_2 \leq \epsilon \\
            \bar{\V}(\x) & \norm{\x}_2 > \epsilon,
        \end{cases}
    \end{equation}
    where the function is a quadratic near the origin with $\bP = \bP^\top \vecdim{n \times n}$ and then, at $\norm{\x}> \epsilon,$ becomes the \ac{cpa} function $\bar{\V} = \{\bar{V}_{\x}\}_{\x \in \mathbb{E}_{\widehat{\Tcal}}}$.
    Consider the optimization problem 
    \begin{equation*}
        \min_{\V, \bP, \alpha, \mathbf{L}, l_p} \alpha
    \end{equation*}
    \vspace{-6.5mm}
    \begin{subequations}\label{eq:optGamma_quad}
        \begin{align}
            &\alpha > 0, \\
            &\bP \succ 0, \\
            &\norm{\bP}_2 \leq l_p, \\
            &\M_{\epsilon} \preceq 0, \\
            &V_\x \geq 0 \quad \forall \x \in \mathbb{E}_{\widehat{\Tcal}}, \\
            &\norm{\nabla V_i}_1 \leq l_i, \quad \forall i \in \mathbb{Z}_1^{m_{\widehat{\Tcal}}}, \\
            &\M_{i,j} \preceq 0, \quad \forall i \in \mathbb{Z}_1^{m_{\widehat{\Tcal}}}, \forall j \in \mathbb{Z}_0^n,
        \end{align}
    \end{subequations}
    where $\alpha = \gamma^2$, $\mathbf{L} = \{l_i\}_{i=1}^{m_{\widehat{\Tcal}}}\subset \mathbb{R}^n,$ and $\M_\epsilon$ is defined as
     \begin{flalign}\label{eq:lmiEp}
         \begin{split}
             &\!\!\!\bmat{ \omega \!\!\!\!\!\! & \bP  \B \!\!\!& \!\!\!\J_h(0)^\top\!\!\! & l_p\epsilon \sum\limits_{k=1}^m\norm{\J_{g_k}(0)}_2 \I & l_p n^{\frac{3}{2}}m^{\frac{1}{2}}\mu_{\epsilon}\epsilon^2\I 
             \\  * \!\!\!\!& (-\frac{\alpha}{2} + \frac{3}{2})\I & \0 & \0 & \0
             \\ * \!\!\!\!& * & -\frac{3}{2}\I & \0 & \0 
             \\ * \!\!\!\!&* &* & -\I & \0
             \\ *\!\!\!\!&* & *& * & -2\I}\!\!.
         \end{split}
     \end{flalign}
     Here, $\omega = \bP\J_f(0) + \J_f(0)^\top\bP + (l_p\epsilon n^{\frac{3}{2}}\beta_{\epsilon} + \frac{1}{2}\epsilon^2 n^2p \rho_{\epsilon}^2)\I,$ $\J_q(0)$ represents the Jacobian of the function, $q$, evaluated at zero, 
     \begin{equation}\label{eq:betaEp}
        \beta_{\epsilon} \geq \max_{p,q,r \in\mathbb{Z}_1^n,\xi \in B_{\epsilon}(0)}\absVal{\frac{\partial^2f^{(p)}(\x)}{\partial \x ^{(q)}\partial \x ^{(r)}}\Bigr|_{\x = \xi}},
    \end{equation}
    \begin{equation}\label{eq:rhoEp}
        \rho_{\epsilon} \geq \max_{p \in\mathbb{Z}_1^p, q,r \in\mathbb{Z}_1^n, \xi \in B_{\epsilon}(0)}\absVal{\frac{\partial^2h^{(p)}(\x)}{\partial \x ^{(q)}\partial \x ^{(r)}}\Bigr|_{\x = \xi}},
    \end{equation}
    and
    \begin{equation}\label{eq:muEp}
        \mu_{\epsilon} \geq \max_{k\in \mathbb{Z}_1^{m}, p,q,r \in\mathbb{Z}_1^n, \xi \in B_{\epsilon}(0)}\absVal{\frac{\partial^2g^{(p)}_{k}(\x)}{\partial \x ^{(q)}\partial \x ^{(r)}}\Bigr|_{\x = \xi}}.
    \end{equation}
    Further, $c_{i,j}$ is given in (\ref{eq:c}), $\beta_i, \rho_{i,a}$ and $\mu_{i,k}$ are given in (\ref{eq:beta}), (\ref{eq:rho}), and (\ref{eq:mu}), and the \ac{lmi} $\M_{i,j}$ is defined by (\ref{eq:gainUB}).

     If the optimization problem, (\ref{eq:optGamma_quad}), is feasible, then (\ref{eq:gain_lmi}) holds for all points in $\Omega^o$ and $\gamma^{*} = \sqrt{\alpha^*}$ is an upper bound on the $\Ell_2$-gain of $\Gcal$ in $\Omega^o.$
\end{theorem}
\begin{proof}
    When $\norm{\x}_2 > \epsilon$ and $V$ is a \ac{cpa} function, the proof follows similarly to \autoref{thm:L2analysisOpt}. When $\norm{\x}_2\leq \epsilon$, Constraint \ref{eq:lmiEp} creates an upper bound on the \ac{hji} for all $\x \in B_{\epsilon}(0)$. To develop this constraint, perform a Taylor series expansion on the \ac{hji} about the origin, $\x^\top\bP\tilde{f} + \tilde{f}^\top\bP\x + \frac{2}{\gamma^2}\x^\top\bP\tilde{g}\tilde{g}^\top \bP\x + \frac{1}{2}\tilde{h}^\top\tilde{h} \leq 0.$
    For example,
    \begin{flalign*}
        &\tilde{g} = g(\x) = \B + \bmat{\J_{g_1}(0)\x & \hdots \J_{g_m}(0)\x} + \\&\frac{1}{2} \bmat{\bmat{\x^\top\bH_{g^{(1)}_1}(\z_{g^{(1)}_1}) \\ \vdots\\\x^\top\bH_{g^{(n)}_1}(\z_{g^{(n)}_1}}\x & \hdots & \bmat{\x^\top\bH_{g^{(1)}_m}(\z_{g^{(1)}_m}) \\ \vdots\\\x^\top\bH_{g^{(n)}_m}(\z_{g^{(n)}_m}}\x},
    \end{flalign*}
    where $\bH_{g_k^{(p)}}(\z_{g^{(p)}_k})$ is the Hessian of the $p$th row and $k$th column of $g$ evaluated at a convex combination of $\x$ and $0$.
    
    Apply \autoref{thm:originLMI} to get the \ac{lmi},
    \begin{align*}
        &\quad \quad\quad \quad\bmat{\tilde{\omega} & \bP\B & \J_h(0)^\top \\ * & (-\frac{\alpha}{2}{+}\frac{3}{2})\I & \0 \\ * & * & -2\I} \preceq 0,\\
    \text{where}\\ 
        \tilde{\omega}=&  \bP\A + \A^\top\bP + (\frac{1}{2}\epsilon^2 n^2p\rho_{\epsilon}^2 {+} \frac{1}{2}l_p^2 m n^3\mu_{\epsilon}^2\epsilon^4 \\&{+}l_p\epsilon n^{\frac{3}{2}}\beta_{\epsilon}{+}l_p^2\sum\limits_{k=1}^m\norm{\J_{g_k}(0)}_2^2\epsilon^2)\I.
    \end{align*}
    Perform a Schur complement \cite{boyd2004convex} about terms containing $l_p^2$ to get (\ref{eq:muEp}).
    
    From \autoref{thm:originLMI}, (\ref{eq:gain_lmi}) holds for all $\x \in B_{\epsilon}(0)$ for $\gamma^*.$ From Theorem 3.1 \cite{james1993partial}, if this discontinuous $V$ is a solution to the \ac{hji}, the $\gamma^*$ is a valid bound on the $\Ell_2$-gain of (\ref{eq:dynSys}).
\end{proof}

After finding an upper bound on a system's gain, Theorem 5 in \cite{lavaei2023L2} can be used to find the small-signal properties of (\ref{eq:dynSys}). This determines the largest value $r_u >0$ so that the system remains within $\Omega^o$ -- where the gain bound is valid.

\section{NUMERICAL EXAMPLE}

Consider the dynamical system
\begin{equation}\label{eq:numExSys}
    \Gcal:
    \begin{cases}
        \xdot_{1}= x_2 \\
        \xdot_2 = -\sin x_1 -x_2 + k(\x)u \\
        y = x_2,
    \end{cases}
\end{equation}
where $\x \vecdim{2},$ $u \in \mathbb{R}$, and $y \in \mathbb{R}.$ Let $\Omega$ be the triangulated region about the origin; an example of which is shown in \autoref{fig:region}. In open-loop, (\ref{eq:numExSys}) has a Lyapunov function, $V(\x) = (1 - \cos x_1) + \frac{1}{2}x_2^2.$ Substituting $V(\x)$ into the \ac{hji} as a storage function results in the inequality,
\begin{equation}\label{eq:hji_numEx}
    0.5x_2^2\left(\gamma^{-1}k(\x)^2 - 1\right) \leq 0.
\end{equation}
Given a $k(\x),$ the $\Ell_2$-gain of $\Gcal$ can be bounded above. 

The numerical examples will show how close the general purpose, algorithmic search proposed here can come to the tight bounds achieved through $V(\x)$, representing ad hoc ingenuity, that cannot necessarily be replicated in all systems. Two variations of $\Gcal$ are considered -- comparing our methodology to analytical techniques of determining $\Ell_2$-gain and previous work that established a system's $\Ell_2$-gain iteratively. In each numerical example, the gain bound is found for increasingly fine triangulations of $\Omega$ using the triangulation refinement process in \cite{lavaei2023synth} with toolbox \cite{engwirda2014locally}.

\begin{figure}
    \centering
    \includegraphics[width =0.9\columnwidth]{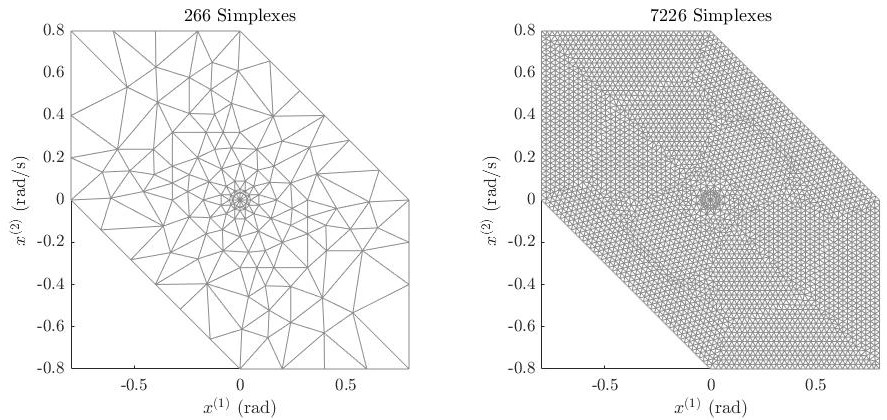}
    \vspace{-2mm}
    \caption{Triangulations of the region, $\Omega$, about the origin for dynamical system, $\Gcal.$ }
    \label{fig:region}
    \vspace{-5mm}
\end{figure}

\subsection{Pendulum}
Consider $\Gcal$ with $k(\x) {=}1$, a classic pendulum. Previous work \cite{lavaei2023L2} could not be used to bound a pendulum's $\Ell_2$-gain, because $k(\x)$ is a linear control affine term. Therefore, $\norm{g(\x)\!^\top \!\!g(\x)}_{\infty}\! {\neq} 0$ at $\x {=} 0$. Theorem (\ref{thm:L2analysisOpt_B}) removes this condition and can be used linear control affine systems.

From (\ref{eq:hji_numEx}), the $\Ell_2$-gain of $\Gcal$ satisfies $\gamma\leq 1$. \autoref{fig:bothEx} shows the result of optimizing (\ref{eq:optGamma_quad}) for an increasing number of $n$-simplexes, as well as the time taken to complete this convex optimization. The $\Ell_2$-gain bound decreases as the triangulation becomes more refined, and reaches its lowest at $\gamma \leq 2.61 $ when opitmizing over 7304 $n$-simplexes. Applying Theorem 5 with Initialization 4 from \cite{lavaei2023L2}, the small-signal property for this system was found to be $||u(t)||_{\infty} \leq 0.077.$

\subsection{Pendulum with Control Affine Input}\label{ex1}
The variation $k(\x) = x_2$ is considered here to compare to \cite{lavaei2023L2}. From (\ref{eq:hji_numEx}), $\gamma \leq \max\{x_2^2|x_2\in\Omega\}$. For the given $\Omega,$ $\gamma \leq 0.64.$ Previous work \cite{lavaei2023L2} developed a non-convex optimization problem and solved it using \ac{ico}, which only guarantees convergence to a local minima, creating more conservative bounds than in (\ref{eq:optGamma}). In \cite{lavaei2023L2}, the gain bound was found to be $\gamma \leq 3.85.$ In comparison, the best $\Ell_2$-gain bound found using the techniques developed in this paper was $\gamma \leq 0.77$ with the small signal property, $||u(t)||_{\infty} \leq 0.2554.$ \autoref{fig:bothEx} shows the optimal gain bound found when solving Problem \ref{eq:optGamma} over different numbers of $n$-simplexes. Even at the lowest number of $n$-simplexes, the gain bound still outperforms \cite{lavaei2023L2} with a bound of $\gamma \leq 1.25.$ 

\begin{figure}
    \centering
    \includegraphics[width =0.95\columnwidth]{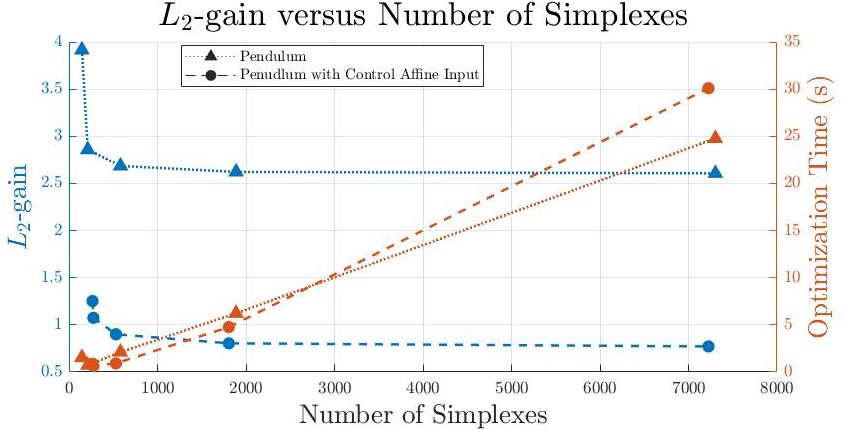}
    \vspace{-2mm}
    \caption{The $\Ell_2$-gain bound was analyzed for an increasing numbers of simplexes over $\Omega$ for two systems --  a pendulum and a pendulum with a control affine input.}
    \label{fig:bothEx}
    \vspace{-5mm}
\end{figure}

\section{DISCUSSION}%

This work developed two convex optimization problems to determine the $\Ell_2$-gain of a dynamical system for a triangulated region about the origin. By reformulating the \ac{hji} as an \ac{lmi} and developing novel \ac{lmi} error bounds for a triangulation, the system's gain can be bounded more tightly than a previous method, \cite{lavaei2023L2}. Moreover, developing an \ac{lmi} to enforce an inequality about the origin enabled a combined quadratic and \ac{cpa} storage function that can be used on systems with linear control affine inputs. A limitation of this methodology is that it can only be applied to bounded regions of the system's state space, so future work should consider methods to expand this technique to the entire state space.

\addtolength{\textheight}{-10cm}   




\section*{ACKNOWLEDGMENT}
Thank you to Dr. Miroslav Krstic for his helpful suggestions on the numerical examples.


\bibliographystyle{IEEEtran}
\bibliography{bstcontrol, IEEEabrv,biblio}
\end{document}

%% file: Packages.tex
\usepackage{amsmath,amssymb,bm} 
\usepackage{acronym}
\usepackage{color}
\usepackage{hyperref}
\usepackage{graphicx}		
\usepackage{mathtools}

\usepackage[normalem]{ulem}
\usepackage{comment}
\usepackage{tikz}

%% file: CustomCommands.tex
\definecolor{mygray}{gray}{0.35}
\newcommand{\black}[1]{\textcolor{black}{#1}} 

\newtheorem{definition}{Definition}[section]

\newcommand{\defeq}{\vcentcolon=}
\newcommand{\defeqRight}{=\vcentcolon}

\newtheorem{theorem}{Theorem}
\newtheorem{lemma}{Lemma}

\newcommand{\I}{\mathbf{I}} 
\newcommand{\mbf}[1]{\mathbf{#1}} 
\newcommand{\bmat}[1]{\begin{bmatrix} #1 \end{bmatrix}} 
\newcommand{\Hcal}{\mathcal{H}} 
\newcommand{\Xcal}{\mathcal{X}}

\newcommand{\A}{\mathbf{A}} 
\newcommand{\B}{\mathbf{B}} 
\newcommand{\x}{\mathbf{x}} 
\newcommand{\bu}{\mathbf{u}} 
\newcommand{\y}{\mathbf{y}} 
\newcommand{\w}{\mathbf{w}} 
\newcommand{\z}{\mathbf{z}} 
\newcommand{\bP}{\mathbf{P}} 
\newcommand{\E}{\mbf{E}} 

\newcommand{\bH}{\mbf{H}}

\newcommand{\M}{\mbf{M}}

\newcommand{\W}{\mbf{W}}
\newcommand{\V}{\mbf{V}}

\newcommand{\X}{\mbf{X}}
\newcommand{\Gcal}{\mathcal{G}}

\newcommand{\Ell}{\mathcal{L}}

\newcommand{\Tcal}{\mathcal{T}}

\newcommand{\vecdim}[1]{\in \mathbb{R}^{#1}}
\newcommand{\xdot}{\dot{x}}

\newcommand{\norm}[1]{\left\lVert#1\right\rVert}
\newcommand{\absVal}[1]{\left\lvert#1\right\rvert}
\newcommand{\0}{\mathbf{0}}
\newcommand{\J}{\mathbf{J}}

\usepackage{centernot}

\acrodef{dl}[{DL}]{deep learning}
\acrodef{rl}[{RL}]{reinforcement learning}
\acrodef{nn}[{NN}]{neural network}
\acrodef{dnn}[{DNN}]{deep neural network}
\acrodef{tdl}[{TDL}]{temporal difference learning}
\acrodef{pid}[{PID}]{proportional–integral–derivative}
\acrodef{us}[{US}]{Ultrasound}
\acrodef{mse}[{MSE}]{mean squared error}
\acrodef{sgd}[{SGD}]{stochastic gradient descent}
\acrodef{ico}[{ICO}]{iterative convex overbounding}
\acrodef{lmi}[{LMI}]{linear matrix inequality}
\acrodef{mjls}[{MJLS}]{Markov jump linear system}
\acrodef{io}[{IO}]{input-output}
\acrodef{iqc}[{IQC}]{integral quadratic constraints}
\acrodef{cnn}[{CNN}]{convolutional neural network}
\acrodef{il}[{IL}]{Imitation learning}
\acrodef{mpc}[{MPC}]{model predictive control}
\acrodef{sdp}[{SDP}]{semi-definite programming}
\acrodef{relu}[{ReLU}]{rectified linear unit}
\acrodef{us}[US]{ultrasound}
\acrodef{mdp}[MDP]{Markov Decision Process}
\acrodef{iid}[iid]{identical and independently distributed random variable}
\acrodef{pid}[PID]{Proportional Integral Derivative}

\acrodef{lqr}[LQR]{linear-quadratic regulator}

\acrodef{cpa}[CPA]{continuous piecewise affine}
\acrodef{hji}[HJI]{Hamilton Jacobi Inequality}